\documentclass{article}
\usepackage[utf8]{inputenc}
\usepackage[letterpaper,margin=1in]{geometry}
\usepackage{amsmath, array, graphicx, amssymb, mathrsfs, amsthm, mathtools}
\usepackage{tikz}
\usepackage{todonotes}
\usepackage{pgfplots}
\usepackage{fontawesome}
\usepackage{enumitem, multicol}
\usepackage{fancyhdr}
\usepackage{wrapfig}
\usepackage[makeroom]{cancel}
\usepackage{booktabs}
\usepackage{comment}
\usepackage{alltt}
\usepackage{pdfpages}
\usepackage{parskip}
\usepackage{hyperref}
\usepackage{caption}
\usepackage{centernot}

\usepackage{adjustbox}

\newcolumntype{P}[1]{>{\centering\arraybackslash}p{#1}}
\newcolumntype{M}[1]{>{\centering\arraybackslash}m{#1}}
\newcolumntype{N}{@{}m{0pt}@{}}
\newcolumntype{L}[1]{>{\raggedright\arraybackslash}m{#1}}
\newcolumntype{R}[1]{>{\raggedleft\arraybackslash}m{#1}}

\makeatletter

\newcommand*{\underarrow}{\def\@underarrow{\relax}\@ifstar{\@@underarrow}{\def\@underarrow{\hidewidth}\@@underarrow}}
\newcommand*{\@@underarrow}[2][]{\underset{\@underarrow\substack{\uparrow\if\relax\detokenize{#1}\relax\else\\#1\fi}\@underarrow}{#2}}

\newcommand*{\overarrow}{\def\@overarrow{\relax}\@ifstar{\@@overarrow}{\def\@overarrow{\hidewidth}\@@overarrow}}
\newcommand*{\@@overarrow}[2][]{\overset{\@overarrow\substack{\if\relax\detokenize{#1}\relax\else#1\\\fi\downarrow}\@overarrow}{#2}}
\makeatother

\pgfplotsset{compat=1.14}

\newcommand{\ds}{\displaystyle}
\tikzset{
  jumpdot/.style={mark=*,solid},
  excl/.append style={jumpdot,fill=white},
  incl/.append style={jumpdot,fill=black},
}

\newtheorem{theorem}{Theorem}[section]
\newtheorem{lemma}[theorem]{Lemma}
\theoremstyle{definition}\newtheorem{remark}[theorem]{Remark}
\theoremstyle{definition}
\theoremstyle{definition}
\theoremstyle{definition}
\theoremstyle{definition}\newtheorem{definition}[theorem]{Definition}
\theoremstyle{definition}\newtheorem{question}[theorem]{Question}
\newtheorem{proposition}[theorem]{Proposition}

\newcommand{\N}{\mathbb{N}}

\newcommand{\C}{\mathbb{C}}
\newcommand{\f}{\frac 12}

\DeclareMathOperator{\orb}{orb}

\usepackage{setspace}
\begin{document}

\title{Ultra hypercyclicity and its connections to mixing properties}
\author{Martin Liu, David Walmsley, James Xue}
\date{\today}
\maketitle

\large

\begin{abstract}
    Recently, two new topological properties for operators acting on a topological vector space were introduced: strong hypercyclicity and hypermixing. We introduce a new property called ultra hypercyclicity and compare it to strong hypercyclicity and hypermixing, as well as the classical notions of mixing, weak mixing, and hypercyclicity. We show that every ultra hypercyclic operator on Fr\'echet space must be weakly mixing, and that there exists a strongly hypercyclic operator which is not ultra hypercyclic. We also characterize, in terms of the weight sequence, the ultra hypercyclic weighted backward shifts on $c_0$ and $\ell^p$, $1\leq p<\infty$. Finally, we improve upon a necessary condition for strongly hypercyclic weighted backward shifts.
\end{abstract}

\section{Introduction}\label{Introduction}
Let $T:X\to X$ be a continuous linear map (henceforth an \textit{operator}) on a separable Fr\'echet space $X$. We are interested in the dynamics, or long-term behavior, of $T$; given some initial point $x\in X$, what can we say about the \textit{orbit of $x$ under $T$}, given by $\orb(T,x)=\{x,Tx,T^2x,T^3x,\ldots\}$? For example, is $x$ \textit{periodic}, meaning $T^n x = x$ for some $n\geq 1$? Or, for quite different behavior, could $\orb(T,x)$ be dense in $X$? When $\orb(T,x)$ is dense, we say $T$ is \textit{hypercyclic} and $x$ is a \textit{hypercyclic vector} for $T$. The presence of a hypercyclic vector and a dense set of periodic points are the two ingredients needed for an operator to be \textit{chaotic}; see \cite[Definition 1.30]{GrossePeris}.

It turns out $\orb(x,T)$ being dense is equivalent to each \textit{return set} $N(x,U)\coloneqq \{n\in \N: T^n x\in U\}$ being non-empty for any non-empty open set $U$ in $X$. With a more topological mindset, we might ask how the iterates of $T$ ``connect" two given non-empty open sets $U$ and $V$ in $X$; given non-empty open subsets $U$ and $V$ of $X$, the \textit{return set} $N(U,V)$ is defined as $N(U,V)=\{n\in \N_0: T^n(U)\cap V\not =\emptyset\}$, where $\N_0=\N\cup \{0\}$. We say $T$ is \textit{topologically transitive} (resp. \textit{weakly mixing}; resp. \textit{mixing}) if for all non-empty open subsets $U$ and $V$ of $X$, $N(U,V)$ is non-empty (resp. contains arbitrarily long intervals; resp. is cofinite). In our setting of a separable Fr\'echet space $X$, it turns out hypercyclicity is equivalent to topological transitivity by the Birkhoff transitivity theorem \cite{Birkhoff1}. Hypercyclicity has become a very active area of research in operator theory that has many connections to other branches of mathematics. For more background on its history, central ideas, and a detailed introduction, we refer the interested reader to the monographs \cite{BayartMatheron,GrossePeris}. 

Given our definitions, and remembering that hypercyclicity and topological transitivity are equivalent, the following implications are immediate:
\begin{center}
    \begin{tabular}{ccccc}
    mixing & $\implies$ & weakly mixing & $\implies$ & hypercyclic. 
\end{tabular}
\end{center}

\noindent Recently in \cite{Ansari1,Ansari2,Ansari3,Ansari4,Curtis}, stronger versions of hypercyclicity and mixing were introduced and studied, namely strong hypercyclicity, supermixing, and hypermixing. We recall these definitions, and introduce the new term ultra hypercyclicity as well.
\begin{definition}
Let $X$ be a separable Fr\'echet space and $T$ be an operator on $X$. 
\begin{itemize}
    \item We say $T$ is \textit{strongly hypercyclic} on $X$ if, for each non-empty open subset $U$ of $X$,
\begin{align*}
    X\setminus \{0\}\subseteq \bigcup_{n=0}^\infty T^n(U).
\end{align*}
    \item We say $T$ is \textit{ultra hypercyclic} if there exists an increasing sequence $(n_k)_{k\geq 1}$ in $\N$ such that, for each non-empty open subset $U$ of $X$,
\begin{align}\label{ultra hyp}
    X = \bigcup_{i=0}^\infty \bigcap_{k=i}^\infty T^{n_k}(U).
\end{align}
When (\ref{ultra hyp}) is satisfied for a given $(n_k)_{k\geq 1}$, we say $T$ is \textit{ultra hypercyclic for $(n_k)_{k\geq 1}$}. And if (\ref{ultra hyp}) is satisfied for $n_k=k$, we say $T$ is \textit{hypermixing}.
    \item We say $T$ is \textit{supermixing} if,  for each non-empty open subset $U$ of $X$,
\begin{align*}
    X = \overline{\bigcup_{i=0}^\infty \bigcap_{n=i}^\infty T^n(U)}.
\end{align*}
\end{itemize}
\end{definition}

\begin{remark}
Because of \cite[Remark 2.5]{Ansari4}, the definition of hypermixing that appears above coincides with the original definition of hypermixing in \cite[Definition 1.1]{Ansari4}.
\end{remark}
The following implications follow quickly from the definitions.

\begin{table}[!h]
\renewcommand{\tablename}{Diagram}
\begin{center}
\caption{The known implications for a continuous operator acting on a Fr\'echet space.}
\label{tab:diagram}
\begin{tabular}{cccccc}
    hypermixing       & $\implies$ &            & supermixing & $\implies$ & mixing\\
    $\Downarrow$      &            &            &             &            & $\Downarrow$\\
    ultra hypercyclic &            &            &             &            & weakly mixing\\
    $\Downarrow$      &            &            &             &            & $\Downarrow$\\
    strongly hypercyclic  &        &            & $\implies$  &            & hypercyclic
\end{tabular}
\end{center}
\end{table}

It is natural to wonder whether any of the implications not present in Diagram \ref{tab:diagram} are true in general. Many are known to fail in general:
\begin{enumerate}[label=\upshape(\roman*)]
    \item That hypercyclicity $\centernot\implies$ weakly mixing is a deep result of De La Rosa and Read \cite{RosaRead}.
    \item \cite[Remark 4.10]{GrossePeris} shows weakly mixing $\centernot\implies$ mixing.
    \item \cite[Example 2.11]{Ansari2} shows that weakly mixing $\centernot\implies$ strongly hypercyclic, and consequently hypercyclicity $\centernot\implies$ strong hypercyclicity.
    \item \cite[Example 4.3]{Ansari4} shows that mixing $\centernot\implies$ strong hypercyclicity and supermixing $\centernot\implies$ hypermixing.
    \item \cite[Theorem 1.4]{Curtis} shows strong hypercyclicity $\centernot\implies$ mixing.
    \item The remarks after \cite[Theorem 2.3]{Curtis} show that mixing $\centernot\implies$ supermixing, since no injective operator can be supermixing.
\end{enumerate}
In this paper, we expand the list above to include the following results:
\begin{enumerate}[label=\upshape(\roman*)]
    \setcounter{enumi}{6}
    \item Ultra hypercyclicity $\implies$ weakly mixing; see Proposition \ref{ultra hyp implies weak mixing}.
    \item Ultra hypercyclicity $\centernot\implies$ mixing; see Theorem \ref{ultra hyp example}.
    \item Strong hypercyclicity $\centernot\implies$ ultra hypercyclicity; see Theorem \ref{diamond example}.
\end{enumerate}

Given these facts, there is only one implication not present in Diagram \ref{tab:diagram} that is unresolved. 
\begin{question}\label{strong hyp imply weak mix}
Must strong hypercyclicity imply the weak mixing property? 
\end{question}
For what it is worth, any irrational circle rotation provides an example of a non-linear dynamical system which is strongly hypercyclic but not weakly mixing; see \cite[Theorem 2.5]{Curtis}. In the linear setting, some progress has been made towards resolving Question \ref{strong hyp imply weak mix}. If $T$ is a strongly hypercyclic operator which is not invertible, then $T$ must satisfy the Hypercyclicity Criterion by \cite[Proposition 1.6]{Ansari3}, and the weak-mixing property is equivalent to satisfying the Hypercyclicity Criterion by an important result from B\'es and Peris \cite[Theorem 2.3]{BesPeris}. Hence every non-invertible strongly hypercyclic operator is weakly mixing. To try and answer Question \ref{strong hyp imply weak mix} in the negative, one then must look for an invertible strongly hypercyclic operator. However, by \cite[Proposition 6]{Ansari1}, $T$ is an invertible strongly hypercyclic operator if and only if every non-zero vector in $X$ is a hypercyclic vector for $T^{-1}$. An operator for which every non-zero vector is a hypercyclic vector is sometimes called \textit{hypertransitive}. Hypertransitive operators are very difficult to construct; the Read operator \cite{Read} is such an example, but it is not surjective. To our knowledge, no known example of an invertible hypertransitive operator exists, and this is one reason why Question \ref{strong hyp imply weak mix} remains open.

\section{Ultra hypercyclicity and weighted backward shifts}\label{backward shifts}

The similarities between the definition of ultra hypercyclic and hypermixing lead immediately to the following ``ultra hypercyclicity criterion." Its proof is entirely similar to the proof of the hypermixing criterion in \cite[Theorem 2.3]{Ansari4}, with the only change in the proof being a replacement of the full sequence $n=1,2,3,\ldots$ with an increasing subsequence $(n_k)_{k\geq 1}$. We omit the details.

\begin{theorem}\label{ultra hyp criterion}
    Suppose $X$ is a first countable topological vector space and $T:X\to X$ is a surjective operator with right inverse map $S$. Then $T$ is ultra hypercyclic for $(n_k)_{k\geq 1}$ if and only if, for every non-zero vector $x\in X$ and any $y\in X$, there exists a sequence $(u_k)_{k\geq 1}$ in $X$ such that for all $k$, $u_k\in \ker T^{n_k}$ and $S^{n_k} x+u_{k}\to y$.
\end{theorem}

As a consequence of the ultra hypercyclicity criterion, no ultra hypercyclic operator $T:X\to X$ is injective. The reasoning is the same as the argument that no hypermixing operator is injective given in \cite[Remark 2.5]{Ansari4}. We include the details for the sake of completeness.

\begin{remark}\label{ultra hyp not injective}
    No ultra hypercyclic operator $T$ is injective. Towards a contradiction, suppose $T$ were an ultra hypercyclic and injective operator with right inverse $S$. Then the kernel of $T$ contains only the zero vector. Consequently, by Theorem \ref{ultra hyp criterion}, there exists an increasing sequence of positive integers $(n_k)_{k\geq 1}$ such that for every non-zero vector $x\in X$ and any $y\in X$, $S^{n_k} x \to y$. Now let $x\in X$ be non-zero and let $y,z\in X$ with $y\not =z$. Then $S^{n_k} x \to y$ and $S^{n_k} x\to z$, which is impossible.
\end{remark}

We are ready to prove that ultra hypercyclicity always implies the weak mixing property for continuous linear operators on Fr\'echet spaces.
\begin{proposition}\label{ultra hyp implies weak mixing}
    If $T:X\to X$ is a continuous, linear, ultra hypercyclic operator on a Fr\'echet space $X$, then $T$ is weakly mixing.
\end{proposition}

\begin{proof}
    Suppose $T:X\to X$ is a continuous, linear, ultra hypercyclic operator on a Fr\'echet space $X$. By Remark \ref{ultra hyp not injective}, $T$ cannot be injective. Then $T$ must be a non-injective strongly hypercyclic operator, which implies $T$ satisfies the Hypercyclicity Criterion by \cite[Proposition 1.6]{Ansari3}, which implies $T$ is weakly mixing by \cite[Theorem 2.3]{BesPeris}.
\end{proof}


To show that ultra hypercyclicity $\centernot \implies$ mixing, and that strong hypercyclicity $\centernot\implies$ ultra hypercyclicity, we turn to the family of weighted backward shifts. Since the orbit of an element with a weighted backward shift can be computed exactly, the family of weighted backward shifts on $c_0$ and $\ell^p$, $1\leq p<\infty$, is often the testing ground for any new notion in linear dynamics. Recall that $c_0$ is the space of all bounded sequences $(x_n)_{n\geq 0}$ in $\C$ for which $x_n\to 0$, and the norm in $c_0$ is given by $\|x\|=\sup_{n\geq 0} |x_n|$. For $1\leq p<\infty$, $\ell^p$ is the space of all complex sequences $(x_n)_{n\geq 0}$ for which $\sum_{n=0}^\infty |x_n|^p <\infty$, and the norm on $\ell^p$ is given by $\|x\|=(\sum_{n=0}^\infty |x_n|^p)^{1/p}$.

Let $X$ be $c_0$ or one of the $\ell^p$ spaces, $1\leq p<\infty$. We denote by $e_n$, $n\geq 0$, the canonical basis vector of $X$ whose only non-zero term is a $1$ in the $n$th position. Every $x=(x_0,x_1,\ldots)$ in $X$ can thus be written as $x=\sum_{i=0}^\infty x_i e_i$. If $w=(w_n)_{n\geq 1}$ is a bounded sequence of positive numbers, the weighted backward shift $B_w$ is defined on the basis of $X$ by $B_w e_0=0$, $B_w e_n = w_n e_{n-1}$ for $n\geq 1$. The operator $B_w$ is continuous on $X$ if and only if $(w_n)_{n\geq 1}$ is bounded, and it is surjective on $X$ if and only if $\inf\{w_n: n\in \N\}>0$; see \cite[Section 1.4.1]{BayartMatheron}. Since surjectivity is a necessary requirement for strong hypercyclicity (see \cite[Proposition 4]{Ansari1}), we consider here only surjective weighted backward shifts. Each weighted backward shift has an associated weighted forward shift $S$ defined on the basis of $X$ by $S e_n = \frac{e_{n+1}}{w_{n+1}}$. When $B_w$ is surjective, $S$ is defined on all of $X$, and the forward shift is a right inverse for $B_w$ on $X$, meaning $B_w S=I$, where $I$ is the identity operator on $X$. 

For simplicity in notations, we put $M_i^j = w_i w_{i+1}\cdots w_{i+j-1}$ for $i,j\geq 1$; that is to say, $M_i^j$ is the product of $j$ consecutive weights, starting with the $i$th weight. Then for any $x\in X$, a formula for $S^n x$ is given by
\begin{align}\label{forward shift}
    S^n x = (0,0,\ldots,0,\frac{x_0}{M_1^{n}},\frac{x_1}{M_2^{n}}, \frac{x_2}{M_3^{n}},\ldots),
\end{align}
where $x_0/M_1^n$ is in the $n$th position. Furthermore, for any $n\in\N$ and $i<j$, by writing out the products in the following expressions, we have
\begin{align}\label{product formula}
    M_i^{n}M_{i+n}^{j-i} = M_i^{n+j-i}=M_i^{j-i} M_j^{n}
\end{align}

Characterizing conditions, in terms of the weight sequence $w=(w_n)_{n\geq 1}$, for $B_w$ to be hypercyclic or mixing are, respectively,
\begin{align*}
    \text{hypercyclic: } \sup_{n\geq 1} M_1^n = \infty, \hskip .25in \text{mixing: } \lim_{n\to \infty} M_1^n =\infty.
\end{align*}
It is a fact that hypercyclicity and the weak-mixing property are equivalent for the family of weighted backward shifts. For these results and much more background on the dynamics of backward shifts, we refer the interested reader to \cite[Section 4.1]{GrossePeris} and the references therein. 


A characterization of hypermixing weighted backward shifts was obtained by Ansari in \cite{Ansari4}. 
\begin{theorem}[{\cite[Theorems 4.1 and 4.4]{Ansari4}}]\label{hypermixing theorem}
    Let $X$ be $c_0$ or $\ell^p$, $1\leq p<\infty$. For a weighted backward shift $B_w$ on $X$, the following are equivalent.
    \begin{enumerate}[label=\upshape(\roman*)]
        \item $B_w$ is hypermixing.
        \item $B_w$ is surjective and $S^n x\to 0$ for every non-zero vector $x\in X$.
        \item $\sup_{n\geq 1} M_1^n = \infty$ and $\inf_{n,k\geq 1} M_n^k >0$.
    \end{enumerate}
\end{theorem}
Towards the goal of characterizing ultra hypercyclic weighted backwards shifts in terms of the weight sequence $(w_n)_{n\geq 1}$, we first record a simple fact about surjective backward shifts.

\begin{lemma}\label{basis vector lemma}
Suppose $B_w$ is surjective on $\ell^p$ or $c_0$, and suppose $(n_k)_{k\geq 1}$ is an increasing sequence of positive integers. Then $S^{n_k} e_j\to 0$ for all $j\in \N_0$ if and only if $S^{n_k}e_j\to 0$ for some $j\in \N_0$. In terms of weights, $M_j^{n_k}\to \infty$ for all $j\in \N_0$ if and only if $M_j^{n_k}\to \infty$ for some $j\in\N_0$.
\end{lemma}

\begin{proof}
    Since $B_w$ is continuous and surjective, $\mu=\ds\sup_{n\geq 1} |w_n|<\infty$ and $\ds\delta=\inf_{n\geq 1} |w_n|>0$. By (\ref{forward shift}), it suffices to show $M_j^{n_k}\to \infty$ for some $j$ implies $M_i^{n_k}\to \infty$ for all $i$. Suppose $M_j^{n_k}\to \infty$ and let $i\not =j$. If $i<j$, then rearranging equation (\ref{product formula}) yields $ M_j^{n_k}=\ds\frac{M_{i+n_k}^{j-i}}{M_i^{j-i}} M_i^{n_k}\leq \left(\frac{\mu}{\delta}\right)^{j-i} M_i^{n_k}$, which proves $M_i^{n_k}\to \infty$. If $j<i$, then equation (\ref{product formula}) yields $\ds M_j^{n_k}=\frac{M_j^{i-j}}{M_{j+n_k}^{i-j}} M_i^{n_k}\leq \left(\frac{\mu}{\delta}\right)^{i-j} M_i^{n_k}$, which proves $M_i^{n_k}\to \infty$.
\end{proof}

We employ the previous two results to characterize ultra hypercyclic weighted backward shifts on $c_0$ and $\ell^p$, $1\leq p<\infty$. While the proof mirrors those in \cite[Theorems 4.1 and 4.4]{Ansari4}, we provide these details for the sake of completeness.

\begin{theorem}\label{UH theorem}
Let $X$ be $c_0$ or $\ell^p$, $1\leq p<\infty$, suppose $B_w$ is a weighted backward shift on $X$, and let $(n_k)_{k\geq 1}$ be an increasing sequence of positive integers. Then the following are equivalent.
\begin{enumerate}[label=\upshape(\roman*)]
        \item  $B_w$ is surjective and is ultra hypercyclic for $(n_k)_{k\geq 1}$.
        \item $B_w$ is surjective and for all non-zero $x\in X$, $S^{n_k}x\to 0$.
        \item $\ds\inf_{n\geq 1} w_n>0$, $\ds\lim_{k\to\infty} M_1^{n_k} = \infty$,  and  $\ds\inf_{i,k} M_i^{n_k} >0$.
    \end{enumerate}
\end{theorem}

\begin{proof}
(i)$\implies$(ii): Assume $B_w$ is surjective and ultra hypercyclic for $(n_k)_{k\geq 1}$. Let $x=(x_0,x_1,\ldots)\in X$ be non-zero. Since the forward shift $S$ is a right inverse for $B_w$ on $X$, Theorem \ref{ultra hyp criterion} implies the existence of a sequence of vectors $(u_k)_{k\geq 1}$ in $\ker B_w^{n_k}$ such that $S^{n_k}x+u_k\to 0$. The vector $u_k$ must have the form $u_k=(u_{k,0}, u_{k,1},\ldots,u_{k,n_k-1},0,0,\ldots)$. Then
\begin{align*}
    S^{n_k}x+u_k=(u_{k,0}, u_{k,1},\ldots,u_{k,n_k-1},\frac{x_0}{M_1^{n_k}}, \frac{x_1}{M_2^{n_k}},\ldots).
\end{align*}
Hence $\|S^{n_k}x\|\leq \|S^{n_k}x+u_k\|\to 0$.

(ii)$\implies$(i): We again use the ultra hypercyclicity criterion. Let $x\in c_0$ be non-zero and let $y=(y_0,y_1,y_2,\ldots)\in c_0$. Then $S^{n_k}x\to 0$, and the sequence of vectors $(u_k)_{k\geq 1}$ given by $u_k=(y_0,y_1,\ldots,y_{n_k-1},0,0,\ldots)$ belongs to $\ker B_w^{n_k}$ and converges to $y$. Thus $S^{n_k}x+u_k\to 0+y=y$, which shows $B_w$ is ultra hypercyclic for $(n_k)_{k\geq 1}$ by Theorem \ref{ultra hyp criterion}.

(ii) $\implies$ (iii): We prove the contrapositive. If $\inf_{n\geq 1} w_n=0$, then $B_w$ is not surjective and thus not ultra hypercyclic, and we are done. If $\ds\lim_{k\to \infty} M_1^{n_k}\not =\infty$, then $M_1^{n_k}$ has a bounded subsequence, say $M_1^{m_k}<\rho$ for each $k$. Then $S^{n_k}e_0$ cannot converge to zero since $\|S^{m_k}e_0\|=(M_1^{m_k})^{-1}>\rho^{-1}$. Hence (ii) is not satisfied by considering $x=e_0$.

So assume that $B_w$ is surjective, meaning $\inf_{n} w_n >0$, and assume  $\ds\lim_{k\to \infty} M_1^{n_k} =\infty$ but $\inf_{i,k} M_i^{n_k}=0$. Then one can inductively construct strictly increasing sequences $(i_l)_{l\geq 1}$ and $(n_{k_l})_{l\geq 1}$ of positive integers such that $M_{i_l}^{n_{k_l}}<2^{-l}$ as follows. Since $\inf_{i,k} M_i^{n_k}=0$, there exist $i_1$ and $n_{k_1}$ such that $M_{i_1}^{n_{k_1}}<2^{-1}$.

Suppose $i_1<\cdots<i_l$ and $n_{k_1}<\cdots < n_{k_l}$ have been chosen so that $M_{i_j}^{n_{k_j}}<2^{-j}$ for each $1\leq j\leq l$. By Lemma \ref{basis vector lemma}, we have $\lim_{k\to \infty} M_j^{n_k}=\infty$ for each $1\leq j\leq i_l$. Hence there exists $K\in \N$ with $K>k_l$ such that $k>K$ implies $M_j^{n_k}>2^{-(l+1)}$ for each $1\leq j\leq i_l$.

Since $\mu=\inf_n w_n>0$, for any $i\in \N$ we have $M_i^{n_k}\geq \mu^{n_k}$ for each $k\in \N$. Let $\epsilon=2^{-1}\min( \{2^{-(l+1)}\}\cup \{\mu^{n_j}: 1\leq j \leq K\})$. Since $\inf_{i,k} M_i^{n_k}=0$, there exist $i$ and $k$ such that $M_i^{n_k}<\epsilon$. Since $M_i^{n_j}\geq \mu^{n_j}$ for $1\leq j\leq K$, and since $M_i^{n_k}<\mu^{n_j}$ for $1\leq j\leq K$, we must have $k>K$. Recalling that $k>K$ implies $M_j^{n_k}> 2^{-(l+1)}$ for $1\leq j\leq i_l$, the fact that $M_i^{n_k}<2^{-(l+1)}$ implies that $i>i_l$. Thus we have shown there must exist $i_{l+1}>i_l$ and $n_{k_{l+1}}>n_{k_l}$ such that $M_{i_{l+1}}^{n_{k_{l+1}}}<2^{-(l+1)}$, as desired.

We now define $x=(x_n)_{n\geq 0}$ by
\[x_n=\begin{cases}
    M_{i_l}^{n_{k_l}} & \text{ if } n=i_l-1 \hskip .1in (l=1,2,3,\cdots)\\
    0 & \text{ otherwise.} 
\end{cases}\]
Then $x\in X$ since $\sum_{l=1}^\infty M_{i_l}^{n_{k_l}}<\sum_{l=1}^\infty 2^{-l}=1$, and the $(i_l+n_{k_l})$th coordinate of the vector $S^{n_{k_l}}x$ is equal to $\frac{M_{i_l}^{n_{k_l}}}{M_{i_l}^{n_{k_l}}}=1$. Hence $S^{n_k}x$ cannot converge to zero as $k\to \infty$, which shows $B_w$ is not ultra hypercyclic for $(n_k)_{k\geq 1}$.

(iii) $\implies$ (ii): Let $X$ be $c_0$ or $\ell^p$, $p\geq 1$, and denote by $\|\cdot\|$ the usual norm on $X$. Let $(n_k)_{k\geq 1}$ be the sequence described in (iii). There exists some $r>0$ for which $\inf_i M_i^{n_k}\geq r$ for all $k$. One can quickly check that for any $x\in X$, $\|S^{n_k} x\| \leq\frac{\|x\|}{r}$.

Let $x=(x_0,x_1,\ldots)\in X$. Let $\epsilon>0$ be given.  There exists $N\in \N$ such that $\|\sum_{i=N+1}^\infty x_i\| < \frac 12 r\epsilon$. Let $x_N = \sum_{i=N+1}^\infty x_i$ and $x_N'=\sum_{i=0}^N x_i$, so that $x=x_N'+x_N$. Then 
\[\|S^{n_k} x_N\|\leq \frac{\|x_N\|}{r}<\frac{\epsilon}{2}.\] 
Furthermore, $M_j^{n_k}\to \infty$ for each $j=1,2,\ldots,N$ by Lemma \ref{basis vector lemma}. Hence there exists $K\in \N$ such that for $k>K$, 
\[M_j^{n_k} \geq \frac{2\|x\|}{\epsilon}.\]
Let $k>K$. Then $\|S^{n_k} x_N'\|\leq \|x_N'\|\frac{\epsilon}{2\|x\|}\leq \|x\|\frac{\epsilon}{2\|x\|}=\frac{\epsilon}{2}$, and the two previous inequalities imply 
\[\|S^{n_k} x\| \leq \| S^{n_k} x_N \| + \| S^{n_k} x_N'\|<\epsilon,\]
which shows $S^{n_k}x\to 0$ as $k\to \infty$.
\end{proof}

\begin{theorem}\label{ultra hyp example}
There exists an ultra hypercyclic weighted backward shift on $c_0$ and $\ell^p$, $1\leq p<\infty$, which is not mixing.
\end{theorem}

\begin{proof}
    We construct a weight sequence $w=(w_n)_{n\geq 1}$ in a recursive fashion using blocks $b_1,b_2,b_3,$ etc. For odd $n$, $b_n$ will be a string of $2$'s of length $\frac{n+1}{2}$.  To describe the rest of our blocks, let $s_n=\sum_{i=1}^n |b_i|$. For even $n$, the block $b_n$ will have the following properties:
\begin{enumerate}[label=\upshape(\roman*)]
    \item   $b_n$ contains the same number of $\f$'s as the number of $2$'s in the previous block, that number being exactly $\frac{n}{2}$;
    \item $b_n$ ends with a $\f$;
    \item the number of $1$'s before each $\f$ in $b_n$ is $s_{n-1}$.
\end{enumerate}
Writing out the weight sequence through the first several blocks can be visualized as so.
\begin{align*}
    \underbrace{2}_{b_1} \underbrace{1 \f}_{b_2 } \underbrace{22 }_{b_3} \underbrace{11111 \f 11111 \f}_{b_4} \underbrace{222}_{b_5}  \underbrace{ 11111111111111111111 \f 111111 \cdots}_{b_6} 
\end{align*}

Clearly $M_1^n =1$ for infinitely many $n$, which implies $B_w$ is not mixing. What remains to show is that $B_w$ is ultra hypercyclic. Since $B_w$ is surjective, by Theorem \ref{UH theorem}(iii), it suffices to show the existence of an increasing sequence $(n_k)_{k\geq 1}$ of positive integers such that $\ds\lim_{k\to\infty} M_1^{n_k} = \infty$,  and  $\ds\inf_{i,k} M_i^{n_k} >0$.

Let $n_k=s_{2k+1}$, so that $n_k$ is the total number of weights in blocks $b_1$ through $b_{2k+1}$. We have arranged the weight sequence so that $M_1^{n_k}=2^{k+1}$. It remains to show $\ds\inf_{i,k} M_i^{n_k} >0$. 

We claim $\ds\inf_{i,k} M_i^{n_k} >\f$. To prove this, let $i,k\in \N$. If $i>s_{2k}$, then $M_i^{n_k}$ is a product of $n_k$ successive weights starting with a weight in block $b_{2k+1}$ or later. Since any instance of $\f$ in the weight sequence past block $b_{2k+1}$ is separated by at least $n_k=s_{2k+1}$ successive $1$'s, the product $M_i^{n_k}$ contains at most one factor of $\f$, and hence $M_i^{n_k}\geq \f$ in the case $i>s_{2k}$.

Now suppose $1\leq i \leq s_{2k}$ and write $M_i^{n_k}=w_i\cdots w_{s_{2k+1}} w_{s_{2k+2}}\cdots w_{n_k+i-1}$. The weights $w_{s_{2k+2}}, \cdots, w_{n_k+i-1}$ all equal $1$, since they are among the first $n_k$ weights in block $b_{2k+2}$. Thus $M_i^{n_k}=w_i\cdots w_{s_{2k+1}}$. Since $n_k$ is the sum of the lengths of blocks $b_1$ through $b_{2k+1}$, if $w_i\cdots w_{s_{2k+1}}$ contains some weight from an even block $b_{2j}$ with $1\leq j < k$, then $w_i\cdots w_{s_{2k+1}}$ will contain each weight from the odd block $b_{2j+1}$. Since the number of $2$'s in block $b_{2j+1}$ is one more than the number of $\f$'s in block $b_{2j}$, we deduce that $M_i^{n_k}\geq 2$ in the case that $1\leq i \leq s_{2k}$, which finishes the proof.
\end{proof}

We now shift our attention to strong hypercyclicity. Ansari proved the following characterization of strong hypercyclicity for weighted backward shifts on $c_0$ and $\ell^p$, $1\leq p<\infty$. Note, however, that this is not a characterization in terms of the underlying weight sequence.

\begin{theorem}[{\cite[Theorem 4.1]{Ansari4}}]\label{strong hyp thm}
Let $X$ be $c_0$ or $\ell^p$, $1\leq p<\infty$, and suppose $B_w$ is a surjective weighted backward shift on $X$. Then $B_w$ is is strongly hypercyclic if and only if for all non-zero $x\in X$, there exists an increasing sequence $(n_k)_{k\geq 1}$ of positive integers such that $S^{n_k}x\to 0$.
  
\end{theorem}

One would hope to expand the above theorem to include a characterization in terms of the weight sequence, similar to Theorem \ref{hypermixing theorem}(iii) and Theorem \ref{UH theorem}(iii). Several necessary conditions for the weight sequence to produce a strongly hypercyclic backward shift are known. First of all, we need $\inf_{n\geq 1} w_n >0$, so that $B_w$ is surjective. Secondly, $B_w$ must be hypercyclic to be strongly hypercyclic, and hence we must have $\sup_{n\geq 1} M_1^n = +\infty$. Another set of necessary conditions was given in the following result.

\begin{proposition}[{\cite[Proposition 2.10]{Ansari2}}]\label{Ansari necessary thm}
    If $B_w$ is strongly hypercyclic on $\ell^p$, $1\leq p<\infty$, then for every increasing sequence $(i_n)_{n\geq 1}$ of positive integers, $\ds\sum_{n=1}^\infty (M_{i_n}^n)^p =\infty$. If $B_w$ is strongly hypercyclic on $c_0$, then for every increasing sequence $(i_n)_{n\geq 1}$ of positive integers, $\ds\lim_{n\to \infty} M_{i_n}^n \not = 0$.
\end{proposition}

We improve the previous result with the following proposition. 

\begin{proposition}\label{c0 necessary theorem}
    If $B_w$ is a strongly hypercyclic weighted backward shift on $\ell^p$ $($respectively on $c_0)$, then $\sum_{n=1}^\infty \big(\inf_i M_i^n)^p=\infty$ $($resp. $\limsup_{n\to \infty} (\inf_i M_i^n) >0)$.
\end{proposition}
\begin{proof}
    \sloppy We prove the contrapositive. Assume that $\ds\sum_{n=1}^\infty \big(\inf_i M_i^n)^p<\infty$ (resp. $\limsup_{n\to \infty} (\inf_i M_i^{n} )=0$). For each $n\in \N$, there exists $i_n\in \N$ such that $M_{i_n}^n < \inf_i M_i^n + 2^{-n}$. Thus $\sum_{n=1}^\infty (M_{i_n}^n)^p < \infty$ (resp. $\lim_{n\to \infty} M_{i_n}^n = 0$). Let $V=\{i_n: n\in \N\}$, and enumerate $V$ in increasing order as $\{v_1,v_2,\ldots\}$.  

    For each $k$, let $N_{v_k}=\{n\in \N: i_n = v_k\}$. Observe $\N$ is the disjoint union of the $N_{v_k}$ sets. Define $x$ in $\ell^p$ (resp. in $c_0$) by 
    \[ x_{i} = \begin{cases}
        \left(\ds\sum_{n\in N_{v_k}} (M_{v_k}^n)^p\right)^{1/p} & \text{ if } i=v_k-1 \hskip .1in (k=1,2,3,\cdots)\\
        0 & \text{ else } 
    \end{cases} \]
    
    \[\left( \text{resp. } x_{i} = \begin{cases}
        \max\{M_{v_k}^n : n\in N_{v_k}\} & \text{ if } i=v_k-1 \hskip .1in (k=1,2,3,\cdots)\\
        0 & \text{ else } 
    \end{cases}\right).\]
    We first prove such an $x$ has the property that $\|S^n x\|\geq 1$, and then we check that the $x$ defined above is actually in $\ell^p$ (resp. in $c_0$).

    Let $n\in \N$.  Then $n$ belongs to some $N_{v_j}$, and the vector $S^{n}x$ contains the term
    \[ \frac{\Big(\ds\sum_{n\in N_{v_j}} (M_{v_j}^n)^p \Big)^{1/p}}{M_{v_j}^n} \hskip .25in \left( \text{resp. } \frac{\max\{M_{v_j}^n : n\in N_{v_j}\}}{M_{v_j}^{n}} \right)\]
    in the $(v_j-1+n)$th coordinate. Since each fraction above is greater than or equal to $1$ in magnitude, we have $\|S^n x\|\geq 1$, which shows $B_w$ cannot be strongly hypercyclic by Theorem \ref{strong hyp thm}.

    It remains to check that the $x$ defined above belongs to $\ell^p$ (resp. $c_0$). For the $\ell^p$ case, since $\N$ is the disjoint union of the $N_{v_k}$ sets, we have
    \[\sum_{i=0}^\infty |x_i|^p =\sum_{k=1}^{\infty}\sum_{n\in N_{v_k}} (M_{v_k}^n)^p = \sum_{n=1}^\infty (M_{i_n}^n)^p <\infty,\]
    which shows $x\in \ell^p$.
    
    Now consider the $c_0$ case. If $V$ is a finite set, then the vector $x$ has a finite number of non-zero terms, which clearly implies $x\in c_0$. So suppose $V$ is infinite. Based on the definition of $x$, it suffices to show $\max\{M_{v_k}^n : n\in N_{v_k}\}\to 0$ as $k\to \infty$. 

    Let $\epsilon>0$. Since $\lim_{n\to \infty} M_{i_n}^n = 0$, we can choose $N\in \N$ such that $M_{i_n}^n < \epsilon$ for all $n>N$. Then if $\min N_{v_k}>N$, it must be that $\max\{M_{v_k}^n : n\in N_{v_k}\}<\epsilon$. Thus it suffices to prove $\min N_{v_k}\to +\infty$ as $k\to \infty$, and this follows from the fact that the $N_{v_k}$ sets are pairwise disjoint. We must have $\min N_{v_k}\not = \min N_{v_j}$ whenever $k\not =j$, and hence the sequence of pairwise distinct positive integers $(\min N_{v_k})_{k\geq 1}$ must diverge to $+\infty$, which finishes the proof.
\end{proof}

To see that Proposition \ref{Ansari necessary thm} is indeed a consequence of Proposition \ref{c0 necessary theorem}, let $(i_n)_{n\geq 1}$ be an increasing sequence of positive integers. Then $M_{i_n}^n \geq \inf_i M_i^n$, so if $\sum_{n=1}^\infty \big(\inf_i M_i^n)^p=\infty$ (resp. $\limsup_{n\to \infty} (\inf_i M_i^n) >0$), then $\ds\sum_{n=1}^\infty (M_{i_n}^n)^p =\infty$ (resp. $\ds\lim_{n\to \infty} M_{i_n}^n \not = 0$).

We next derive a sufficient condition for a weighted backward shift on $\ell^p$ or $c_0$ to be strongly hypercyclic. 
\begin{proposition}\label{sufficent for SH}
    Suppose $B_w$ is a weighted backward shift on $\ell^p$ or $c_0$, and there exists $a>0$ and a function $f:(0,a)\to (0,\infty)$ with $\lim_{x\to 0^+} f(x)=+\infty$ such that for all $\epsilon \in (0,a)$ and for all $N\in \N$, there exists $n\in \N$ such that the following conditions hold:
    \begin{enumerate}[label=\upshape(\alph*)]
        \item $M_i^n > f(\epsilon)$ for all $i\leq N$, and
        \item $\epsilon<M_i^n$ for all $i$.
    \end{enumerate}
    Then $B_w$ is strongly hypercyclic.
\end{proposition}

\begin{proof}
    Let $x=(x_0,x_1,x_2,\ldots)$ be a non-zero element of $\ell^p$ (resp. of $c_0$). We must show there exists an increasing sequence $(n_k)_{k\geq 1}$ in $\N$ such that $\|S^{n_k} x\|\to 0$. Let $\epsilon_k=a/(k+1)$. Since $\epsilon_k\to 0$ as $k\to \infty$, we have $\frac{\|x\|}{f(\epsilon_k)} + \epsilon_k\to 0$. Hence it suffices to show that for each $k\in\N$, there exists an $n$, arbitrarily large, such that 
    \begin{align}\label{sufficientinduction}
        \|S^{n}x\| \leq \frac{\|x\|}{f(\epsilon_k)} + \epsilon_k.
    \end{align}
    
    Since $x\in \ell^p$ (resp. $x\in c_0$), there exists a value $N\in \N$ such that
    \begin{align*}
        \Big(\sum_{i>N}^\infty|x_{i-1}|^p\Big)^{1/p} < (\epsilon_k)^2 \hskip .25in \big(\text{resp. }i > N \Rightarrow |x_{i-1}|\leq (\epsilon_k)^2 \text{ if } x\in c_0 \big).
    \end{align*}
    Then by our assumptions, there exists $n\in \N$ such that
    \begin{enumerate}
        \item[(A)] $M_i^{n}> f(\epsilon_k)$ for all $i\leq N$, and
        \item[(B)] $\epsilon_k<M_i^{n}$ for all $i$.
    \end{enumerate}
    Then inequality (A) implies 
    \[\Big(\sum_{i=1}^{N}\frac{|x_{i-1}|^p}{M_i^{n}}\Big)^{1/p} \leq \frac{1}{f(\epsilon_k)} \Big(\sum_{i=1}^{N}|x_{i-1}|^p\Big)^{1/p}\leq \frac{\|x\|}{f(\epsilon_k)} \hskip .25in \big( \text{resp. } i\leq N\Rightarrow \frac{|x_{i-1}|}{M_i^{n}}\leq \frac{\|x\|}{f(\epsilon_k)}  \big).\]
    Furthermore, (B) implies
    \[\Big(\sum_{i>N}\frac{|x_{i-1}|^p}{M_i^{n}}\Big)^{1/p}\leq \frac{1}{\epsilon_k} \Big(\sum_{i>N}|x_{i-1}|^p\Big)^{1/p}\leq \frac{(\epsilon_k)^2}{\epsilon_k}=\epsilon_k  \big( \text{resp. } i>N \Rightarrow \frac{|x_{i-1}|}{M_i^{n}}\leq \frac{(\epsilon_k)^2}{\epsilon_k}=\epsilon_k  \big).\]
    Since 
    \[S^{n}x = (0,0,\ldots,0,\frac{x_0}{M_1^{n}},\frac{x_1}{M_2^{n}},\ldots, \frac{x_{N-1}}{M_{N}^{n}},\frac{x_{N}}{M_{N+1}^{n}}, \frac{x_{N+1}}{M_{N+2}^{n}},\ldots),\]
    combining the previous two inequalities yields that, whether $x\in \ell^p$ or $x\in c_0$,
    \begin{align*}
        \|S^{n} x\| & \leq \left\| \left(\frac{x_0}{M_1^{n}},\frac{x_1}{M_2^{n}},\ldots, \frac{x_{N-1}}{M_{N}^{n}},0,0,\ldots\right)\right\| + \left\| \left(0,\cdots,0,\frac{x_{N}}{M_{N+1}^{n}}, \frac{x_{N+1}}{M_{N+2}^{n}},\ldots\right)\right\| \\
        & \leq \frac{\|x\|}{f(\epsilon_k)} + \epsilon_k,
    \end{align*}
    as desired.
\end{proof}

With the sufficient condition in Proposition \ref{sufficent for SH}, we are able to provide an example of a strongly hypercyclic weighted backward shift $B_w$ on $\ell^p$ and $c_0$ which is not ultra hypercyclic. We are grateful to Fedor Petrov who generously provided the ideas that led to the discovery of the following example.
\begin{theorem}\label{diamond example}
    There exists an ultra hypercyclic weighted backward shift on $c_0$ and $\ell^p$, $1\leq p<\infty$, which is strongly hypercyclic but neither ultra hypercyclic nor mixing.
\end{theorem}
\begin{proof}
Consider the weight sequence $w=(w_n)_{n\geq 1}$ defined as follows: if $n$ is odd then $w_n=2$, and if $n$ is even then $w_n=(2w_{\frac{n}{2}})^{-1}$. In other words, $$w_n=\begin{cases}
        2 & \text{ if } n \text{ is odd},\\
        2 & \text{ if } n \text{ is even and } w_{\frac{n}{2}}=\frac{1}{4},\\
        \frac{1}{4} & \text{ if } n \text{ is even and } w_{\frac{n}{2}}=2.
    \end{cases}$$

First, we develop several properties that this weight sequence has. It is immediate from the definition of $w_n$ that, for any $i\in \N$,
\begin{equation}\label{twoinverse}
    w_{2i-1}w_{2i}=\frac{1}{w_i}.
\end{equation}
By repeatedly applying this identity to each factor in $(M_i^k)^{-1}$, for any $i,k\in \N$, we have
\[(M_i^k)^{-1}=\frac{1}{w_i}\cdots \frac{1}{w_{i+k-1}}=w_{2i-1}w_{2i}\cdots w_{2i+2k-3}w_{2i+2k-2}=M_{2i-1}^{2k}.\]
Furthermore, since $2i+2k-1$ and $2i-1$ are both odd, it follows that
\[M_{2i}^{2k}=\frac{w_{2i+2k-1}}{w_{2i-1}}M_{2i-1}^{2k}=\frac{2}{2}M_{2i-1}^{2k}=M_{2i-1}^{2k}\]
for any $i,k\in \N$.  Combining these last two observations yields, for any $i,k\in \N$,
\begin{align}\label{mtwoinverse1}
    \frac{1}{M_{i}^k} & = M_{2i-1}^{2k},\\ \label{mtwoinverse2}
    \frac{1}{M_{i}^k} & = M_{2i}^{2k}.
\end{align}

We claim these two previous identities yield the following four identities for any $i,k\in \N$:
\begin{equation}\label{mfourinverse}
    M_i^k=M_{4i-3}^{4k}=M_{4i-2}^{4k}=M_{4i-1}^{4k}=M_{4i}^{4k}.
\end{equation}
We must check that $M_i^k=M_{4i-2a-b}^{4k}$ for any $a,b\in\{0,1\}$. Indeed, one of the identities (\ref{mtwoinverse1}) and (\ref{mtwoinverse2}) gives $M_i^k=(M_{2i-a}^{2k})^{-1}$, and then by applying one of (\ref{mtwoinverse1}) and (\ref{mtwoinverse2}) to $M_{2i-a}^{2k}$, we also have $(M_{2i-a}^{2k})^{-1}=M_{2(2i-a)-b}^{2(2k)}=M_{4i-2a-b}^{4k}$. This proves all four identities in (\ref{mfourinverse}).

We now prove, by induction on $m\in\N$, the identity
\begin{equation}\label{lengthtwo}
    M_j^{4^mk}=M_i^k\quad\text{for }j\in\{4^m(i-1)+1,\dots,4^mi\}
\end{equation}
for any $i,k\in \N$. The basis for induction follows immediately from (\ref{mfourinverse}). For the induction step, assume (\ref{lengthtwo}) holds for each positive integer less than or equal to $m$, and let $j\in \{4^{m+1}(i-1)+1,\dots,4^{m+1} i\}$. There exist integers $l_1\in \{1,\ldots,4^m\}$ and $l_2\in \{0,1,2,3\}$ such that $j=4^{m+1}(i-1)+4l_1-l_2$. Let $j'=4^m(i-1)$. Since $l_2\in\{0,1,2,3\}$, one of the identities in (\ref{mfourinverse}) implies $M_{j'+l_1}^{4^m k}=M_{4(j'+l_1)-l_2}^{4^{m+1}k}=M_{j}^{4^{m+1}k}$. And by the induction hypothesis, since $j'+l_1\in \{4^m(i-1)+1,\dots,4^mi\}$, we have $M_{j'+l_1}^{4^m k}=M_i^k$. By (\ref{mfourinverse}) it follows that $M_j^{4^{m+1}k}=M_i^k$, which completes the induction step and proves (\ref{lengthtwo}) holds for any $m,i,k\in \N$.

We are now ready to prove $B_w$ is strongly hypercyclic but not ultra hypercyclic. We set $a_1=1$, and define $a_k=4a_{k-1}+1$ for each $k\ge 2$. We claim that for every $k\in\N$, \begin{equation}\label{fact1}
    M_1^{a_k}=2^k.
\end{equation} 
To see this, we induct on $k$. The base case $k=1$ is clear. Let $k\ge 2$ and assume $M_1^{a_i}=2^i$ for every positive integer $i<k$. Since $a_k$ is odd, by definition of $w_n$, we have $w_{a_k}=2$. Applying (\ref{mfourinverse}) with $i=1$ and $a_{k-1}$ in place of $k$, we also have $M_1^{4a_{k-1}}=M_1^{a_{k-1}}$. Hence 
\[M_1^{a_k}=M_1^{a_k-1}\cdot w_{a_k}=M_1^{4a_{k-1}}\cdot 2=M_1^{a_{k-1}}\cdot 2=2^{k-1}\cdot 2=2^k.\]
This finishes the induction step and proves (\ref{fact1}). 

We also claim for every $k\in\N$,
\begin{equation}\label{fact2}
    \inf_iM_i^{a_k}\ge4^{-k}.
\end{equation}
Again we induct on $k$, and the base case is clear from definition of $w_n$. Let $k\ge 2$ and assume $\inf_iM_i^{a_j}\ge4^{-j}$ for every positive integer $j<k$. Since $M_i^{a_k}=M_i^{a_k-1}w_{i+a_k-1}$ for every positive integer $i$, since $\inf_i w_i=\frac{1}{4}$, and since $a_k-1=4a_{k-1}$, we have 
\[\inf_iM_i^{a_k}\ge\inf_i\frac{1}{4}M_i^{a_k-1}=\inf_i\frac{1}{4}M_i^{4a_{k-1}}.\] 
Observe that (9) implies for every $i\in\mathbb{N}$, there exists some $j\in\mathbb{N}$ such that $M_i^{4a_{k-1}}=M_j^{a_{k-1}}$. Hence it follows that
\[\inf_i\frac{1}{4}M_i^{4a_{k-1}}\ge\inf_i\frac{1}{4}M_i^{a_{k-1}}\geq \frac{1}{4}\cdot\frac{1}{4^{k-1}},\]
where the last inequality follows from the induction hypothesis. This proves (\ref{fact2}).

Next we verify the conditions of Proposition \ref{sufficent for SH} are satisfied, which will imply that $B_w$ is strongly hypercyclic. Define $a=1/4$ and $f:(0,a)\to(0,\infty)$ by $f(x)=(4\sqrt{x})^{-1}$. It is clear that $\lim_{x\to 0^+}f(x)=\infty$. Let $\epsilon\in(0,a)$ and let $N\in\N$. Choose a positive integer $k$ such that $2^{k-1}\le(4\sqrt{\epsilon})^{-1}<2^{k}$, and choose a positive integer $m$ such that $4^m\ge N$. Let $n=4^ma_{k}$. By (\ref{lengthtwo}) and (\ref{fact1}) we have:
\[M_i^{n}=2^k>\frac{1}{4\sqrt{\epsilon}}=f(\epsilon)\quad\text{for }i\in\{1,2,\dots,4^m\}.\]

Let $i\in\N$. Choose $j\in\N$ such that $i\in\{4^m(j-1)+1,\dots,4^mj\}$. Note that $2^{k-1}\le(4\sqrt{\epsilon})^{-1}$ implies $\epsilon\le4^{-(k+1)}<4^{-k}$. Hence by (\ref{lengthtwo}) and (\ref{fact2}) we have
$$M_i^n=M_j^{a_k}\ge\frac{1}{4^k}>\epsilon.$$
Since $i\in \N$ was arbitrary, the conditions of Proposition \ref{sufficent for SH} are satisfied, which proves $B_w$ is strongly hypercyclic.

Now we show that $B_w$ is not ultra hypercyclic by showing condition (iii) of Theorem \ref{UH theorem} is not satisfied. Note that (\ref{mtwoinverse1}) implies for every $n\in\N$, $$\frac{1}{M_1^{n}}=M_1^{2n}=M_1^nM_{n+1}^n, $$ so for every $n\in\N$,
\begin{equation*}\label{petrov}
    M_{n+1}^{n}=\frac{1}{(M_1^{n})^2}.
\end{equation*}
Therefore, if $(n_k)_{k\geq 1}$ is such that $\lim_{k\to \infty}M_1^{n_k}=\infty$, then we must also have $\lim_{k\to \infty}M_{n_k+1}^{n_k}=0$. Hence if $(n_k)_{k\geq 1}$ satisfies $\lim_{k\to \infty}M_1^{n_k}=\infty$, then we must have $\inf_{i,k} M_i^{n_k}=0.$ Thus $B_w$ cannot be ultra hypercyclic by Theorem \ref{UH theorem}, nor can it be mixing since $\lim_{n\to \infty} M_1^n \not = \infty$.
\end{proof}

For the family of weighted backward shifts, every property in Diagram \ref{tab:diagram} other than strong hypercyclicity has a characterization in terms of properties the weight sequence has. We invite the reader to try and characterize the strongly hypercyclic weighted backward shifts in terms of the weight sequence, and we ask whether the necessary conditions in Theorem \ref{c0 necessary theorem} are also sufficient for strong hypercyclicity on $\ell^p$ and $c_0$.

\begin{question}
    Suppose $B_w$ is a surjective, hypercyclic weighted backward shift on $\ell^p$, $1\leq p<\infty$ (resp. on $c_0$). If $\ds\sum_{n=1}^\infty \big(\inf_i M_i^n)^p=\infty$ (resp.  $\ds\limsup_{n\to \infty} \big(\inf_{i\geq 1} M_i^{n} \big)>0$), must $B_w$ be strongly hypercyclic?
\end{question}

As pointed out by the referee, one could also explore the size of the set of ultra hypercyclic weighted backward shifts with respect to the strongly hypercyclic ones in a topological sense. For example, Chan and Sanders showed in \cite{ChanSanders} that the chaotic backward shifts are a dense subset of the set of all backward shifts on $\ell^2$ in the strong operator topology. In fact, there is an SOT-dense path of such operators. One might ask similar questions about ultra hypercyclic shifts.

\begin{question}\label{ref question}
    Is the set of ultra hypercyclic weighted backward shifts SOT-dense in the set of strongly hypercyclic weighted backward shifts on $\ell^2$? Is it a $G_\delta$ subset?
\end{question}

\section{Acknowledgments}
We are grateful to Fedor Petrov for providing the ideas that led to Example \ref{diamond example}. We also thank the St. Olaf College Center for Undergraduate Research and Inquiry for supporting the project. We also thank the referee for pointing out Question \ref{ref question}, and for providing many helpful suggestions which improved the presentation of the results.
\bibliography{Bib}

\end{document}